\documentclass[11pt, twoside, leqno]{article}

\usepackage{amssymb}
\usepackage{amsmath}
\usepackage{amsthm}
\usepackage{amsfonts}
\usepackage{mathrsfs}
\usepackage{color}

\allowdisplaybreaks

\def\rr{{\mathbb R}}
\def\rn{{\mathbb{R}^n}}
\def\urn{\mathbb{R}_+^{n+1}}
\def\zz{{\mathbb Z}}
\def\cc{{\mathbb C}}

\def\nn{{\mathbb N}}

\def\cf{{\mathcal F}}

\def\cm{{\mathcal M}}

\def\cp{{\mathcal P}}

\def\cs{{\mathcal S}}

\def\hbf{\emph{\textbf{F}}}
\def\hbg{\emph{\textbf{G}}}

\def\fz{\infty }
\def\az{\alpha}

\def\ez{\epsilon}

\def\lz{\lambda}

\def\sz{\sigma}
\def\vz{\varphi}
\def\vez{\varepsilon}

\def\lf{\left}
\def\r{\right}

\def\hs{\hspace{0.25cm}}
\def\ls{\lesssim}

\def\noz{\nonumber}

\def\st{\subset}

\def\loc{{\mathop\mathrm{\,loc\,}}}

\def\vlp{{L^{p(\cdot)}(\rn)}}
\def\vhs{H^{p(\cdot)}(\rn)}

\newtheorem{theorem}{Theorem}[section]
\newtheorem{lemma}[theorem]{Lemma}
\newtheorem{corollary}[theorem]{Corollary}
\newtheorem{proposition}[theorem]{Proposition}

\theoremstyle{definition}
\newtheorem{remark}[theorem]{Remark}
\newtheorem{definition}[theorem]{Definition}

\renewcommand{\appendix}{\par
   \setcounter{section}{0}%
   \setcounter{subsection}{0}%
   \setcounter{subsubsection}{0}%
   \gdef\thesection{\@Alph\c@section}%
   \gdef\thesubsection{\@Alph\c@section.\@arabic\c@subsection}%
   \gdef\theHsection{\@Alph\c@section.}%
   \gdef\theHsubsection{\@Alph\c@section.\@arabic\c@subsection}%
   \csname appendixmore\endcsname
 }

\numberwithin{equation}{section}

\begin{document}

\arraycolsep=1pt

\title{\bf\Large Characterizations of Variable Exponent Hardy Spaces
via Riesz Transforms
\footnotetext{\hspace{-0.35cm} 2010 {\it
Mathematics Subject Classification}. Primary 42B30;
Secondary 47B06, 42B35, 42B25.
\endgraf {\it Key words and phrases.} Hardy space,
variable exponent, Riesz transform, harmonic function
\endgraf This project is supported by the National
Natural Science Foundation of China
(Grant Nos.~11171027 and 11361020),
the Specialized Research Fund for the Doctoral Program of Higher Education
of China (Grant No. 20120003110003) and the Fundamental Research
Funds for Central Universities of China
(Grant Nos.~2013YB60 and 2014KJJCA10). The third author is supported by
Grant-in-Aid for Scientific Research (B), No. 15H03621,
Japan Society for the Promotion of Science.}}
\author{Dachun Yang, Ciqiang Zhuo\footnote{Corresponding author}\ \ and Eiichi Nakai}
\date{}
\maketitle

\vspace{-0.8cm}

\begin{center}
\begin{minipage}{13cm}
{\small {\textbf{ Abstract}}\quad
Let $p(\cdot):\ \mathbb R^n\to(0,\infty)$ be a variable exponent function
satisfying that there exists a constant $p_0\in(0,p_-)$, where
$p_-:=\mathop{\mathrm {ess\,inf}}_{x\in \mathbb R^n}p(x)$, such that the
Hardy-Littlewood maximal operator is bounded on the variable exponent Lebesgue space
$L^{p(\cdot)/p_0}(\mathbb R^n)$. In this article, via investigating relations
between boundary valued of harmonic functions on the upper half space
and elements of variable exponent Hardy spaces $H^{p(\cdot)}(\mathbb R^n)$
introduced by E. Nakai and Y. Sawano and, independently, by D. Cruz-Uribe and
L.-A. D. Wang,
the authors characterize $H^{p(\cdot)}(\mathbb R^n)$ via the first order
Riesz transforms when $p_-\in (\frac{n-1}n,\infty)$, and via compositions
of all the first order Riesz transforms when $p_-\in(0,\frac{n-1}n)$.}
\end{minipage}
\end{center}

\vspace{-0.1cm}

\section{Introduction\label{s1}}
\hskip\parindent
The main purpose of this article is to establish Riesz transform characterizations
of the variable exponent Hardy spaces introduced by  Nakai and Sawano \cite{ns12}
and, independently, by Cruz-Uribe and Wang \cite{cw14}.
Let $\cs(\rn)$ be the \emph{space of all Schwartz functions} on $\rn$.
Recall that, for all $j\in\{1,\,\dots,\,n\}$, the \emph{$j$-th Riesz transform}
is usually defined by setting, for all $f\in\cs(\rn)$ and $x\in\rn$,
\begin{equation*}
R_j(f)(x):=\lim_{\delta\to0^+}C_{(n)}
\int_{\{y\in\rn:\ |y|>\delta\}}\frac{y_j}{|y|^{n+1}}f(x-y)\,dy,
\end{equation*}
here and hereafter, $\delta\to 0^+$ means that $\delta\in(0,\fz)$ and $\delta\to 0$,
$C_{(n)}:=\frac{\Gamma([n+1]/2)}{\pi^{(n+1)/2}}$ and $\Gamma$ denotes
the Gamma function.

It is well known that the Riesz transform is a natural generalization of the
Hilbert transform on the real line $\rr$ to the higher dimension
Euclidean space $\rn$. Moreover, Riesz transforms are the most typical examples
of Calder\'on-Zygmund operators and have many interesting and useful properties
(see, for example, \cite{Gra1,stein70,stein93} and their references).
Indeed, they are the simplest, non-trivial and ``invariant" operators under
the acting of the group of rotations in the Euclidean space $\rn$,
and also constitute typical and important examples of Fourier multipliers.
Recall also that, when studying the boundedness of Riesz transforms on
Lebesgue spaces $L^p(\rn)$ with $p\in(0,1]$, the Hardy space $H^p(\rn)$ was introduced
as a suitable substitute of $L^p(\rn)$
and now plays very important roles in harmonic analysis and partial differential
equations (see, for example, \cite{CoWe77,FeSt72,muller94}).

Beside the boundedness of Riesz transforms or, more generally, Calder\'on-Zygmund type
operators on various function spaces including Hardy spaces, the Riesz transform
characterizations on Hardy spaces or local Hardy spaces also attract much attention
(see, for example, \cite{ccyy15, FeSt72, Gol79, PeSe08, stein93,Uch01, Wh76}).
The research on characterizations of Hardy spaces via Riesz transforms
originated from Fefferman-Stein's
celebrating seminal paper \cite{FeSt72} in 1972 and then was extended by Wheeden
\cite{Wh76} to the weighted Hardy space $H_w^1(\rn)$. Very recently,
Cao et al. \cite{ccyy15} established Riesz transform characterizations of
some Hardy spaces of Musielak-Orlicz type introduced by Ky \cite{Ky14},
which essentially extends the results of Wheeden \cite{Wh76}
and also the corresponding results of Hardy spaces $H^p(\rn)$
(see, for example, \cite[p.\,123, Proposition 3]{stein93}).
As was well known, when establishing Riesz transform characterizations
of Hardy spaces $H^p(\rn)$,
one needs to extend the elements of $H^p(\rn)$ to the upper half space
$\urn:=\rn\times(0,\fz)$ via Poisson integrals. This extension in turn has a close
relationship with the analytical definition of $H^p(\rn)$ which is the key
starting point for the study on the Hardy space, before one paid attention to
the real-variable theory of $H^p(\rn)$ (see \cite{MuWh78,stein61,StWe60,StWe68}).

It is known that the Hardy space $H^p(\rn)$ can be characterized by the first order
Riesz transforms when $p\in(\frac{n-1}n,\fz)$ (see \cite{FeSt72} or
\cite[p.\,123, Proposition 3]{stein93})
and by the higher order Riesz transforms when $p\in(0,\frac{n-1}n]$
(see \cite[p.\,168]{FeSt72}). Precisely, denote by $\cs'(\rn)$ the \emph{dual space}
of $\cs(\rn)$, namely, the \emph{space of all tempered distributions}.
Recall that a distribution $f\in\cs'(\rn)$ is called a \emph{distribution
restricted at infinity} if there exists a positive number $r$ large enough such that,
for all $\phi\in\cs(\rn)$, $f\ast \phi\in L^r(\rn)$. Then the following
conclusions hold true, which were obtained in \cite{FeSt72} (see also
\cite[p.\,123, Proposition 3 and p.\,133, Item 5.16]{stein93}).

\begin{theorem}[\cite{FeSt72}]\label{t-1.1}
Let $\phi\in\cs(\rn)$ satisfy $\int_\rn\phi(x)\,dx=1$.

{\rm(i)} If $p\in(\frac{n-1}n,\fz)$, then $f\in H^{p}(\rn)$ if and only if
$f$ is a distribution restricted at infinity and there
exists a positive constant $A_1$ such that, for all $\ez\in(0,\fz)$,
$$\|f\ast \phi_\ez\|_{L^p(\rn)}+\sum_{j=1}^n\|R_j(f)\ast\phi_\ez\|_{L^p(\rn)}\le A_1,$$
where, for all $x\in\rn$, $\phi_\ez(x):=\frac1{\ez^n}\phi(\frac x\ez)$;
moreover, $\|f\|_{H^p(\rn)}\sim A_1$ with implicit equivalent positive constants independent of
$f$ and $\ez$.

{\rm(ii)} If $m\in\nn$ and $p\in(\frac{n-1}{n+m-1},\fz)$,
then $f\in H^p(\rn)$ if and only if $f$ is a distribution
restricted at infinity and there exists a positive constant $A_2$ such that,
for all $\ez\in(0,\fz)$,
$$\|f\ast \phi_\ez\|_{L^p(\rn)}+
\sum_{k=1}^m\sum_{j_1,\,\dots,\,j_k=1}^n
\|R_{j_1}\cdots R_{j_k}(f)\ast\phi_\ez\|_{L^p(\rn)}\le A_2;$$
moreover, $\|f\|_{H^p(\rn)}\sim A_2$ with implicit equivalent positive constants independent of
$f$ and $\ez$.
\end{theorem}

In this article, we are devoted to extending conclusions of Theorem
\ref{t-1.1} to the variable exponent Hardy space $\vhs$
which was first studied by Nakai and Sawano \cite{ns12}
and, independently, by Cruz-Uribe and Wang \cite{cw14}. Recall that the \emph{variable
exponent Lebesgue space $\vlp$}, with a variable exponent function $p(\cdot):\ \rn\to(0,\fz)$,
consists of all measurable functions $f$ such that $\int_\rn|f(x)|^{p(x)}\,dx<\fz$.
As a generalization of classical Lebesgue spaces, variable exponent
Lebesgue spaces were introduced by Orlicz \cite{or31} in 1931, however, they
have been the subject of more intensive study since the early 1990s, because
of their intrinsic interest for applications into harmonic analysis, partial
differential equations and variational integrals with nonstandard growth conditions
(see \cite{cfbook,dhr11,ins14} and their references).
The variable exponent Hardy space $\vhs$ extends not only the variable exponent
Lebesgue space $L^{p(\cdot)}(\rn)$ but also the Hardy space $H^p(\rn)$ with constant
exponent.

Particularly, Nakai and Sawano \cite{ns12} introduced Hardy spaces with variable
exponents, $H^{p(\cdot)}(\rn)$, and established their atomic characterizations
which were further applied to consider their duals.
Later, in \cite{Sa13}, Sawano extended the
atomic characterization of the space $H^{p(\cdot)}(\rn)$ in \cite{ns12},
which also improves the corresponding result in \cite{ns12},
and gave out more applications including the boundedness of several operators.
Moreover, Zhuo et al. \cite{zyl14} established their equivalent characterizations
via intrinsic square functions, including the intrinsic Lusin area function,
the intrinsic $g$-function and the intrinsic $g_\lz^\ast$-function.

Independently,
Cruz-Uribe and Wang \cite{cw14} also investigated the variable exponent Hardy
space $\vhs$ with $p(\cdot)$ satisfying some slightly weaker conditions than those used in \cite{ns12}. Moreover,
in \cite{cw14}, characterizations of $\vhs$ in terms of radial or non-tangential
maximal functions or atoms were established, and the boundedness of singular integral
operators were also obtained.

A measurable function $p(\cdot):\ \rn\to(0,\fz)$ is called a
\emph{variable exponent}. For any variable exponent $p(\cdot)$, let
$$p_-:=\mathop{\rm ess\,inf}_{x\in \rn}p(x)
\quad {\rm and}\quad
p_+:=\mathop{\rm ess\,sup}_{x\in \rn}p(x).$$
Denote by $\cp(\rn)$ the \emph{collection of variable exponents}
$p(\cdot)$ \emph{satisfying} $0<p_-\le p_+<\fz$.

For a measurable function $f$ on $\rn$ and a variable exponent $p(\cdot)\in\cp(\rn)$,
the \emph{modular} $\varrho_{p(\cdot)}(f)$ of $f$
is defined by setting
$\varrho_{p(\cdot)}(f):=\int_\rn|f(x)|^{p(x)}\,dx$ and the
\emph{Luxemburg} (also called \emph{Luxemburg-Nakano}) \emph{quasi-norm} is given by
\begin{equation*}
\|f\|_{\vlp}:=\inf\lf\{\lz\in(0,\fz):\ \varrho_{p(\cdot)}(f/\lz)\le1\r\}.
\end{equation*}
Then the \emph{variable exponent Lebesgue space} $\vlp$ is defined to be the
set of all measurable functions $f$ such that $\varrho_{p(\cdot)}(f)<\fz$
equipped with the quasi-norm $\|f\|_{\vlp}$.

\begin{remark}\label{r-vlp}
 Let $p(\cdot)\in\cp(\rn)$.

(i) If $p_-\in[1,\fz)$, then $L^{p(\cdot)}(\rn)$
is a Banach space (see \cite[Theorem 3.2.7]{dhr11}).
Particularly, for all $\lz\in\cc$ and $f\in\vlp$,
$\|\lz f\|_{\vlp}=|\lz|\|f\|_{\vlp}$ and, for all $f,\ g\in\vlp$,
$$\|f+g\|_{\vlp}\le \|f\|_{\vlp}+\|g\|_{\vlp}.$$

(ii) For any non-trivial function $f\in \vlp$, it holds true that
$\varrho_{p(\cdot)}(f/\|f\|_{\vlp})=1$ (see, for example, \cite[Proposition 2.21]{cfbook}).

(iii) If $\int_\rn[|f(x)|/\delta]^{p(x)}\,dx\le c$ for some $\delta\in(0,\fz)$
and some positive constant $c$ independent of $\delta$, then it is easy to see that
$\|f\|_{\vlp}\le C\delta$, where $C$ is a positive constant independent of $\delta$,
but depending on $p_-$ (or $p_+$) and $c$.

(iv) (\textbf{The H\"older inequality}) Assume that
$p_-\in[1,\fz)$.
It was proved in \cite[Lemma 3.2.20]{dhr11} (see also \cite[Theorem 2.26]{cfbook})
that, if $f\in \vlp$ and $g\in L^{p^\ast(\cdot)}(\rn)$, then $fg\in L^1(\rn)$
and
$$\int_\rn|f(x)g(x)|\,dx\le C\|f\|_{\vlp}\|g\|_{L^{p^\ast(\cdot)}(\rn)},$$
where $p^\ast(\cdot)$ denotes the \emph{dual variable exponent} of $p(\cdot)$ defined by
$\frac1{p(x)}+\frac 1{p^\ast(x)}=1$ for all $x\in\rn$, with $1/\fz=0$,
and $C$ is a positive constant
depending on $p_-$ or $p_+$, but independent of $f$ and $g$.

(v) Assume that $p_-\in[1,\fz)$.
For any given measurable function $f$ on $\rn$, let
$$\|f\|_{\vlp}^\ast:=\sup \lf|\int_\rn f(x)g(x)\,dx\r|,$$
where the supremum is taken over all $g\in L^{p^\ast(\cdot)}(\rn)$ satisfying
$\|g\|_{L^{p^\ast(\cdot)}(\rn)}\le1$. Then it was proved in \cite[Theorem 2.34]{cfbook}
that, for all $f\in\vlp$, $\|f\|_{\vlp}\sim \|f\|_{\vlp}^\ast$ with
the implicit equivalent positive constants independent of $f$.
\end{remark}

For all $r\in(0,\fz)$, denote by $L_\loc^r(\rn)$ the \emph{set of all locally
$r$-integrable functions} on $\rn$ and, for any measurable set $E\st \rn$,
by $L^r(E)$ the \emph{set of all measurable functions $f$ such that}
$$\|f\|_{L^r(E)}:=\lf\{\int_E|f(x)|^r\,dx\r\}^{1/r}<\fz.$$
Recall that the \emph{Hardy-Littlewood maximal operator $\cm$} is defined by setting,
for all $f\in L_\loc^1(\rn)$ and $x\in\rn$,
$$\cm(f)(x):=\sup_{B\ni x}\frac1{|B|}\int_B|f(y)|\,dy,$$
where the supremum is taken over all balls $B$ of $\rn$ containing $x$.

For any $p(\cdot)\in\cp(\rn)$, $p(\cdot)$ is said to belong to $\cm\cp(\rn)$
if there exists $p_0\in(0,p_-)$ such that the maximal operator $\cm$
is bounded on $L^{p(\cdot)/p_0}(\rn)$.

For any $N\in\nn$, let
$$\cf_N(\rn):=\lf\{\psi\in\cs(\rn):\ \sup_{\az,\,\beta\in\zz_+^n,\,|\az|,\,|\beta|\le N}
\sup_{x\in\rn}|x^\alpha \partial^\beta\psi(x)|\le1\r\},$$
where, for all $\beta:=(\beta_1,\dots,\beta_n)\in\zz_+^n$,
$|\beta|:=\beta_1+\cdots+\beta_n$ and
$\partial^\beta:=(\frac{\partial}{\partial x_1})^{\beta_1}
\cdots (\frac{\partial}{\partial x_n})^{\beta_n}$.
In what follows, for all $\psi\in\cs(\rn)$, $t\in(0,\fz)$ and $\xi\in\rn$,
let $\psi_t(\xi):=t^{-n}\psi(\xi/t)$.
Then, for all $f\in\cs'(\rn)$, the \emph{radial grand maximal function}
$f^\ast_{N,+}$ is defined by setting, for all $x\in\rn$,
\begin{equation}\label{granmf}
f^\ast_{N,+}(x):=\sup_{\psi\in\cf_N(\rn)}M^\ast_{\psi,+}(f)(x),
\end{equation}
where, $M^\ast_{\psi,+}(f)$ denotes the \emph{radial maximal function} of $f$ defined
by setting, for all $x\in\rn$,
$$M^\ast_{\psi,+}(f)(x):=\sup_{t\in(0,\fz)}|f\ast \psi_t(x)|.$$

\begin{definition}
Let $p(\cdot)\in \cm\cp(\rn)$ with $p_0\in(0,p_-)$ and
$N\in (\frac n{p_0}+n+1,\fz)$. Then the \emph{variable
exponent Hardy space $\vhs$} is defined to be the set of all
$f\in\cs'(\rn)$ such that $f^\ast_{N,+}\in\vlp$ endowed with the quasi-norm
$\|f\|_{\vhs}:=\|f^\ast_{N,+}\|_{\vlp}$.
\end{definition}

\begin{remark}\label{r-defn}
(i) We point out that the variable exponent Hardy space was first studied
by Nakai and Sawano \cite{ns12}. However, in \cite{ns12}, the variable exponent
$p(\cdot)\in\cp(\rn)$
is required to satisfy \emph{globally H\"older continuous condition}, namely,
there exist positive constants $C_{\log}(p)$ and $C_\fz$, and $p_\fz\in\rr$ such that,
for all $x$, $y\in\rn$,
$$|p(x)-p(y)|\le \frac{C_{\log}(p)}{\log(e+1/|x-y|)}$$
and
$$|p(x)-p_\fz|\le \frac{C_\fz}{\log(e+|x|)}.$$

(ii) Independently, Cruz-Uribe and Wang \cite{cw14}
also studied the variable exponent Hardy space, but they only assume that
$p(\cdot)\in\cm\cp(\rn)$.
We point out that,
if $p(\cdot)$ satisfies the globally H\"older continuous condition,
then $p(\cdot)\in \cm\cp(\rn)$; see \cite{cfbook,dhr11}.
Therefore, the assumption on $p(\cdot)$ in \cite{cw14} is slightly
weaker than that in \cite{ns12}. In the present article, we only
assume that $p(\cdot)\in\cm\cp(\rn)$.

(iii) Let $p(\cdot)\in\cp(\rn)$. If the maximal operator $\cm$
is bounded on $L^{p(\cdot)}(\rn)$, then, for all $s\in(1,\fz)$, $\cm$ is
also bounded on $L^{sp(\cdot)}(\rn)$ (see \cite[Lemma 2.12]{cw14}).
\end{remark}

The main results of this article are stated as follows.

\begin{theorem}\label{t-rtc1}
Let $p(\cdot)\in \cm\cp(\rn)$ with $p_-\in(\frac{n-1}n,\fz)$, $f\in\cs'(\rn)$
and $\phi\in\cs(\rn)$
satisfy $\int_\rn\phi(x)\,dx=1$. Then the following items are equivalent:

{\rm(i)} $f\in\vhs$;

{\rm(ii)} $f$ is a distribution restricted at infinity and there exists a positive
constant $A_3$ such that, for all $\ez\in(0,\fz)$,
\begin{equation}\label{rtc-3}
\|f\ast\phi_\ez\|_{\vlp}+\sum_{j=1}^n\|R_j(f)\ast\phi_\ez\|_{\vlp}\le A_3.
\end{equation}

Moreover, there exists a positive constant $C$, independent of $f$,
such that
$$C^{-1}\|f\|_{\vhs}\le A_3\le C\|f\|_{\vhs}.$$

Furthermore, if $p_-\in[1,\fz)$, then \eqref{rtc-3} can be replaced by
\begin{equation}\label{rtc-3a}
\|f\|_{\vlp}+\sum_{j=1}^n\|R_j(f)\|_{\vlp}\le A_3.
\end{equation}
\end{theorem}

\begin{theorem}\label{t-rtc2}
Let $m\in\nn\cap [2,\fz)$ and $p(\cdot)\in \cm\cp(\rn)$ satisfy
$p_-\in(\frac{n-1}{n+m-1},\fz)$,
and let $f\in\cs'(\rn)$ and $\phi\in\cs(\rn)$
satisfy $\int_\rn\phi(x)\,dx=1$.
Then the following items are equivalent:

{\rm(i)} $f\in \vhs$;

{\rm(ii)} $f$ is a distribution restricted at infinity and
there exists a positive constant $A_4$ such that, for all $\vez\in(0,\fz)$,
$$\|f\ast\phi_\vez\|_\vlp+\sum_{k=1}^m\sum_{j_1,\,\dots,\,j_k=1}^n
\|R_{j_1}\cdots R_{j_k}(f)\ast\phi_\vez\|_\vlp\le A_4.$$

Moreover, there exists a positive constant $C$, independent of $f$, such that
$$C^{-1}\|f\|_{\vhs}\le A_4\le C\|f\|_{\vhs}.$$
\end{theorem}

\begin{remark}\label{r-1.7x}
(i) Obviously, when $p(\cdot)\equiv p$ with $p\in(\frac{n-1}n,\fz)$,
Theorem \ref{t-rtc1} is just Theorem \ref{t-1.1}(i)
(see also \cite[p.\,123, Proposition 3]{stein93}) and, when $p(\cdot)\equiv p$ with
$p\in(\frac{n-1}{n+m+1},\fz)$, Theorem \ref{t-rtc2} is just
Theorem \ref{t-1.1}(ii) (see also \cite[p.\,122, Item 5.16]{stein93}).

(ii) Observer that, for any $p(\cdot)\in\cp(\rn)$, there exists an $m\in\nn$ such that
$p_-\in(\frac{n-1}{n+m-1},\fz)$. Thus, in Theorem \ref{t-rtc2},
we actually obtain the Riesz transform characterization of $\vhs$
for all $p(\cdot)\in\cm\cp(\rn)$ with $p_-\in(0,\fz)$ via different choices of $m\in\nn$.

(iii) If $f$ is a distribution restricted at infinity, then,
for any $j\in\{1,\dots,n\}$, $R_j(f)$ is well defined as a distribution
(see \cite[p.\,123]{stein93}).

(iv) Recall that, in \cite{ccyy15}, Cao et al. characterized the Musielak-Orlicz
Hardy space $H^\vz(\rn)$, which was introduced by Ky \cite{Ky14},
via Riesz transforms. Observe that, if
\begin{equation}\label{mo-f}
\vz(x,t):=t^{p(x)}\quad{\rm for\ all}\quad x\in\rn\quad{\rm and}\quad t\in(0,\fz),
\end{equation}
then $H^\vz(\rn)=H^{p(\cdot)}(\rn)$. However, a general Musielak-Orlicz function
$\vz$ satisfying all the assumptions in \cite{Ky14} (and hence \cite{ccyy15})
may not have the form as in \eqref{mo-f}. On the other hand, it was proved in
\cite[Remark 2.23(iii)]{yyz13} that there exists a variable exponent
$p(\cdot)$ satisfying the globally H\"older continuous condition and hence belonging to
$\cm\cp(\rn)$, but $t^{p(\cdot)}$ is not a uniformly Muckenhoupt weight, which
was required in \cite{Ky14} (and hence \cite{ccyy15}).
Therefore, the Musielak-Orlicz Hardy space $H^\vz(\rn)$ in \cite{Ky14}
(and hence \cite{ccyy15}) and the variable exponent Hardy space $\vhs$ in \cite{cw14}
(and hence in the present article) can not cover each other.
\end{remark}

This article is organized as follows.

In Section \ref{s2}, we aim at giving out the proofs of Theorems \ref{t-rtc1}
and \ref{t-rtc2}.
To this end, via borrowing some ideas used in the proof
of Theorem \ref{t-1.1}, we not only make full use of
the non-tangential and the radial maximal function
characterizations of $\vhs$, but also draw supports from properties of
harmonic functions.

Precisely, in Subsection \ref{s2.1}, we first recall some known results on $\vhs$
including its characterizations
in terms of the non-tangential maximal function corresponding to Poisson kernels
in Lemma \ref{l-vhs} below (see also \cite[Theorem 3.1]{cw14})
and the non-tangential grand maximal function
in Lemma \ref{l-vhs1} below (see also \cite[Proposition 2.1]{zyl14}).
Using these characterizations, we investigate relations between harmonic functions
and Poisson integrals of elements of $\vhs$
in Proposition \ref{p-hc} below, which are further applied in Lemma \ref{l-hv} below
to find the harmonic majorant of harmonic vectors satisfying
the generalized Cauchy-Riemann equation \eqref{gcr-equ} below.
Moreover, the boundary value of such a harmonic vector in turn determines
a harmonic function $u$
such that its non-tangential maximal function belongs to $\vlp$ (see Proposition
\ref{p-hv1} below). To prove Proposition \ref{p-hv1}, we need to show that
the set $\vhs\cap L^{1+r}(\rn)$ with $r\in[p_+,\fz)$ is dense in $\vhs$ via
the Poisson kernels,
which improves \cite[Proposition 4.2]{cw14} in which Cruz-Uribe and Wang
proved that $\vhs\cap L_{\loc}^1(\rn)$ is dense in $\vhs$
and may be of independent interest.
The boundedness of Riesz transforms on $\vhs$, which is essentially established
by Cruz-Uribe and Wang in \cite{cw14}, also plays a key role in the proof of Theorem
\ref{t-rtc1}.

In Subsection \ref{s2.2}, by some arguments similar to those used in
the proof of Theorem \ref{t-rtc1}, we prove Theorem \ref{t-rtc2} via tensor-valued
functions.
We point out that one of important facts used in the proofs of Theorems \ref{t-rtc1}
and \ref{t-rtc2} is that, for all $f\in \vlp$,
$$\int_\rn \lf[|f(x)|/\|f\|_\vlp\r]^{p(x)}\,dx=1,$$
if $0<p_-\le p_+<\infty$,
which is well known (see \cite[Proposition 2.21]{cfbook}). Another important
fact used in the proofs of Theorems \ref{t-rtc1}
and \ref{t-rtc2} is that, for any $f\in\vhs$ and $t\in(0,\fz)$, the Poisson integral
$f\ast P_t\in L^r(\rn)$ as long as $r\in[p_+,\fz)$, which is proved in
Lemma \ref{l-dense}(ii) below.

Different from the strategy used in the case of Musielak-Orlicz Hardy spaces
in \cite{ccyy15}, we do not need to restrict the working space to be $\vhs\cap L^2(\rn)$
to guarantee that $\{R_j(f)\}_{j=1}^n$ make sense.
Indeed, for any $f\in\vhs$, we can prove that $f$ is a distribution restricted at
infinity (see Lemma \ref{l-dense}(i)), which is not clear whether it is true for elements of
Musielak-Orlicz Hardy spaces or not.
Thus, for any $j\in\{1,\,\dots,\,n\}$ and $f\in\vhs$, $R_j(f)$
is well defined as a distribution according to Remark \ref{r-1.7x}(iii).

Finally, we make some conventions on notation. Let $\nn:=\{1,2,\dots\}$ and $\zz_+:=\nn\cup\{0\}$.
We denote by
$C$ a \emph{positive constant} which is independent of the main
parameters, but may vary from line to line. The \emph{symbol}
$A\ls B$ means $A\le CB$. If $A\ls B$ and $B\ls A$, then we write $A\sim B$.
If $E$ is a subset of $\rn$, we denote by $\chi_E$ its
\emph{characteristic function} and by $E^\complement$ the set $\rn\backslash E$.
For all $r\in(0,\fz)$ and $x\in\rn$, let $B(x,r):=\{y\in\rn:\ |x-y|<r\}$.

\section{Proofs of main results\label{s2}}
\hskip\parindent
This section is devoted to the proofs of of Theorems \ref{t-rtc1} and
\ref{t-rtc2}. Precisely, Theorem \ref{t-rtc1} is proved in Subsection \ref{s2.1}
and the proof of Theorem \ref{t-rtc2} is presented out in Subsection \ref{s2.2}.

\subsection{The case $p_-\in(\frac{n-1}n,\fz)$\label{s2.1}}
\hskip\parindent
In this subsection, we prove Theorem \ref{t-rtc1}. To this end,
we need more preparations and begin with the following notation.

A distribution $f\in\cs'(\rn)$ is called a \emph{bounded distribution}
if, for all $\psi\in\cs(\rn)$, $f\ast \psi\in L^\fz(\rn)$.
For a bounded distribution $f$, its \emph{non-tangential maximal function}
is defined by setting, for all $x\in\rn$,
$$
f_P^\ast (x):=\sup_{(y,t)\in\urn,\,|x-y|<t}|f\ast P_t(y)|,
$$
where $P_t$ denotes the \emph{Poisson kernel}, which is defined
by setting, for all $(x,t)\in\rr_+^{n+1}$,
\begin{equation}\label{po-ke}
P_t(x)
:=\frac{\Gamma([n+1]/2)}{\pi^{(n+1)/2}}\frac t{(t^2+|x|^2)^{(n+1)/2}}.
\end{equation}

\begin{remark}\label{r-welld}
It is known that,
if $f$ is a bounded distribution, then $f\ast P_t$ is well defined.
Indeed, by \cite[p.\,90]{stein93}, we see that there exist
$\psi^{(1)},\ \psi^{(2)}\in\cs(\rn)$ and $h\in L^1(\rn)$ such that, for all $t\in(0,\fz)$,
\begin{equation}\label{poisson}
P_t=\lf(\psi^{(1)}\r)_t\ast h_t+\lf(\psi^{(2)}\r)_t.
\end{equation}
Thus, if $f$ is a bounded distribution, then
$f\ast P_t=f\ast (\psi^{(1)})_t\ast h_t+f\ast (\psi^{(2)})_t$
and hence $f\ast P_t$ is well defined.
Moreover, the function $u(x,t):=f\ast P_t(x)$ for $(x,t)\in\urn$ is harmonic on $\urn$
(see \cite[p.\,90]{stein93}).
\end{remark}

The following conclusion was obtained in \cite[Theorem 3.1]{cw14}.

\begin{lemma}\label{l-vhs}
Let $p(\cdot)\in \cm\cp(\rn)$ and $f\in \cs'(\rn)$. Then
the following three items are equivalent:

{\rm(i)} $f\in \vhs$;

{\rm (ii)} there exists $\Phi\in\cs(\rn)$ with $\int_\rn \Phi(x)\,dx\neq0$ such that
$M^\ast_{\Phi,+}(f)\in\vlp$;

{\rm(iii)} $f$ is a bounded distribution and $f_P^\ast \in\vlp$.

Moreover,
$$\|f\|_{\vhs}\sim \|M^\ast_{\Phi,+}(f)\|_{\vlp}\sim \|f_P^\ast \|_{\vlp},$$
where the implicit equivalent positive constants are independent of $f$.
\end{lemma}

\begin{remark}\label{r-vhs}
If $f\in \vhs$, then, by Lemma \ref{l-vhs}(iii), $f$ is a bounded distribution.
Thus, for any $f\in \vhs$, $f\ast P_t$ is well defined by
Remark \ref{r-welld}.
\end{remark}

For all $N\in\nn$ and $f\in\cs'(\rn)$, the
\emph{non-tangential grand maximal function}, $f_{N,\triangledown}^\ast$, of $f$
is defined by setting, for all $x\in\rn$,
$f_{N,\triangledown}^\ast(x):=\sup_{\psi\in\cs_N(\rn)}M_{\psi,\triangledown}^\ast(f)(x),$
where $M_{\psi,\triangledown}^\ast(f)$ denotes the \emph{non-tangential maximal function} of $f$
defined by setting, for all $y\in\rn$,
\begin{equation}\label{ntmf}
M_{\psi,\triangledown}^\ast(f)(y):=\sup_{t\in(0,\fz),\,|\xi-y|<t}|f\ast\psi_t(\xi)|.
\end{equation}

We have the following equivalent characterization of $\vhs$ via
non-tangential grand maximal functions, whose proof can be found in
\cite[Proposition 2.1]{zyl14} (see also \cite[Lemma 7.9]{cw14}).

\begin{lemma}\label{l-vhs1}
Let $p(\cdot)\in \cm\cp(\rn)$ with $p_0\in(0,p_-)$ and $N\in(\frac n{p_0}+n+1,\fz)$.
Then $f\in\vhs$ if and only if $f\in\cs'(\rn)$ and $f_{N,\triangledown}^\ast\in\vlp$; moreover,
there exists a positive constant $C$ such that, for all $f\in\vhs$,
$$C^{-1}\|f_{N,\triangledown}^\ast\|_{\vlp}\le \|f\|_{\vhs}\le
C\|f_{N,\triangledown}^\ast\|_{\vlp}.$$
\end{lemma}

Let $u$ be a measurable function on $\urn$. Then its \emph{non-tangential maximal
function $u^\ast$} is defined by setting, for all $x\in\rn$,
\begin{equation*}
u^\ast(x):=\sup_{t\in(0,\fz),\,|y-x|<t}|u(y,t)|.
\end{equation*}

Recall that a function $u$ on $\urn$ is said to be \emph{harmonic} if,
for all $t\in(0,\fz)$ and $x:=(x_1,\dots,x_n)\in\rn$,
$\partial_t^2 u(x,t)+\sum_{j=1}^n\partial_{x_j}^2u(x,t)=0.$

On relations between harmonic functions and Poisson integrals of
elements of the variable exponent Hardy space $\vhs$, we have the following conclusion.

\begin{proposition}\label{p-hc}
Let $p(\cdot)\in \cm\cp(\rn)$ and $u$ be a harmonic function on $\urn$.
Then $u^\ast\in \vlp$ if and only if there exists $f\in \vhs$ such that,
for all $(x,t)\in\urn$, $u(x,t)=f\ast P_t(x)$. Moreover, there exists a positive
constant $C$, independent of $f$ and $u$, such that
$$C^{-1}\|f\|_{\vhs}\le \|u^\ast\|_{\vlp}\le C\|f\|_{\vhs}.$$
\end{proposition}

\begin{proof}
Assume that there exists $f\in \vhs$ such that $u(x,t)=f\ast P_t(x)$ for all
$(x,t)\in\urn$. Then, by Lemma \ref{l-vhs}, we see that
$f_P^\ast \in\vlp$, which, together with the fact that
$u^\ast= f_P^\ast $, implies that $u^\ast\in \vlp$
and $\|u^\ast\|_\vlp\ls \|f\|_{\vhs}$.

Conversely, suppose that $u^\ast\in \vlp$.
We first claim that $u$ is bounded in any half space $\{(x,t)\in\urn:\ t\ge\vez>0\}$.
We may assume that $\|u^*\|_{\vlp}>0$.
Indeed, observe that, for any $(x,t)\in\urn$ and $y\in B(x,t)$,
$|u(x,t)|\le u^\ast (y)$. Then, by Remark \ref{r-vlp}(ii),
we find that, for all $(x,t)\in\urn$ satisfying $|u(x,t)|\ge \|u^*\|_{\vlp}$,
\begin{eqnarray*}
|u(x,t)|^{p_-}
&&= \frac 1{|B(x,t)|}
\int_{B(x,t)}\lf[\frac{|u(x,t)|}{\|u^*\|_{\vlp}}\r]^{p_-}\|u^*\|_{\vlp}^{p_-}\,dy\\
&&\ls t^{-n} \|u^*\|_{\vlp}^{p_-}\int_{B(x,t)}\lf[\frac{u^{\ast}(y)}{\|u^*\|_{\vlp}}\r]^{p(y)}\,dy
\ls t^{-n} \|u^*\|_{\vlp}^{p_-},
\end{eqnarray*}
which shows that, for all $(x,t)\in\urn$ with $t\ge \vez>0$,
$$
 |u(x,t)|\ls (1+\vez^{-n/p_-})\|u^\ast\|_{\vlp}.
$$
Thus, the above claim holds true.

Now, we let $\vez\in(0,\fz)$ be fixed and $f_\vez(\cdot):=u(\cdot,\vez)$.
We prove that, for all $(x,t)\in \urn$,
\begin{equation}\label{2.5x1}
f_\vez\ast P_t(x)=u(x, t+\vez)=:u_\vez(x,t).
\end{equation}
Indeed, from the above claim, \cite[p.\,48, Theorem 2.1(b)]{StWe71} and
the fact $f_\vez$ is hounded and continuous,
we deduce that $f_\vez\ast P_t$
is bounded, harmonic on $\urn$, continuous on the closure
$\overline{\urn}:=\rn\times[0,\fz)$ and equals to $f_\vez$ on the boundary
$\rn\times\{0\}$. On the other hand,
by the assumption on $u$ and the above claim again, we know that $u(\cdot, t+\vez)$ is also bounded,
harmonic on $\urn$, continuous on the closure
$\overline{\urn}:=\rn\times[0,\fz)$ and has the same value with $f_\vez\ast P_t$
on the boundary $\rn\times\{0\}$.
Thus, by the maximum modulus principle of harmonic functions
(see \cite[p.\,40,\ Theorem 1.5]{StWe71}), we conclude that \eqref{2.5x1} holds true.

Since $u_\vez^\ast\le u^\ast\in \vlp$, it follows, from Lemma \ref{l-vhs}, that
\begin{equation}\label{hc-x1}
\|f_\vez\|_{\vhs}\sim \|u_\vez^\ast\|_{\vlp}\ls\|u^\ast\|_{\vlp}.
\end{equation}
This implies that $\{f_\vez\}_{\vez\in(0,\fz)}$ is uniformly bounded in $\vhs$
and hence in $\cs'(\rn)$ due to \cite[Remark 3.5]{ns12}.
Therefore, by the weak compactness of $\cs'(\rn)$
(see, for example, \cite[p.\,119]{StWe71}),
we conclude that there exist a subsequence $\{\vez_k\}_{k\in\nn}\st(0,\fz)$,
$\vez_k\to 0$ as $k\to\fz$
and an element $f\in\cs'(\rn)$ such that $f_{\vez_k}\to f$ as $k\to\fz$ in $\cs'(\rn)$.
Moreover, for any $\Phi\in\cs(\rn)$ with $\int_\rn\Phi(x)\,dx\neq0$
and $x\in\rn$, we have
\begin{equation*}
M^\ast_{\Phi,+}(f)(x):=\sup_{t\in(0,\fz)}|f\ast\Phi_t(x)|
= \sup_{t\in(0,\fz)}\lim_{k\to\fz}|f_{\vez_k}\ast\Phi_t(x)|
\le \lim_{k\to\fz}M^\ast_{\Phi,+}(f_{\vez_k})(x),
\end{equation*}
which, combined with Lemma \ref{l-vhs}, the Fatou lemma (see \cite[Theorem 2.61]{cfbook})
and \eqref{hc-x1}, implies that
\begin{eqnarray*}
\|f\|_{\vhs}
&&\sim\|M^\ast_{\Phi,+}(f)\|_{\vlp}\ls\lf\|\lim_{k\to\fz}M^\ast_{\Phi,+}(f_{\vez_k})\r\|_{\vlp}\\
&&\ls \liminf_{k\to\fz}\lf\|M^\ast_{\Phi,+}(f_{\vez_k})\r\|_{\vlp}
\sim \liminf_{k\to\fz}\|f_{\vez_k}\|_{\vhs}\ls \|u^\ast\|_{\vlp}.
\end{eqnarray*}
This shows that $f\in \vhs$ and hence $f\ast P_t$ is well defined due to
Remark \ref{r-welld}.
By this, \eqref{poisson}, the facts that $u$ is harmonic
and $\{f_{\vez_k}\}_{k\in\nn}$ converges to $f$ in $\cs'(\rn)$,
we find that, for all $(x,t)\in\urn$,
$$u(x,t)=\lim_{k\to\fz}u(x,t+\vez_k)=
\lim_{k\to\fz}(f_{\vez_k}\ast P_t)(x)=(f\ast P_t)(x),$$
which completes the proof of Proposition \ref{p-hc}.
\end{proof}

Let $\hbf:=\{u_0,u_1,\dots,u_n\}$ be a vector of harmonic functions
on $\urn$. Then
$\hbf$ is said to satisfy the \emph{generalized Cauchy-Riemann equation} if, for all
$j,\,k\in\{0,1,\dots,n\}$,
\begin{equation}\label{gcr-equ}
\lf\{
\begin{array}{rl}
\sum\limits_{j=0}^n\dfrac{\partial u_j}{\partial x_j}=0,\\
\dfrac{\partial u_j}{\partial x_k}=\dfrac{\partial u_k}{\partial x_j},
\end{array}\r.
\end{equation}
where, for any $(x,t)\in \urn$, we let $x:=(x_1,\dots,x_n)$ and $x_0:=t$.

For the harmonic vector satisfying \eqref{gcr-equ}, we have the following
conclusion.

\begin{proposition}\label{p-hv1}
Let $p(\cdot)\in \cm\cp(\rn)$ with $p_-\in(\frac{n-1}n,\fz)$
and $u$ be a harmonic function on $\urn$. Then
$u^\ast\in \vlp$ if and only if there exists a harmonic vector
$\textbf{F}:=\{u_0,u_1,\dots,u_n\}$ satisfying \eqref{gcr-equ}, $u_0=u$
and
\begin{equation}\label{hbf-norm}
\sup_{t\in(0,\fz)}\||\textbf{F}(\cdot,t)|\|_{\vlp}<\fz,
\end{equation}
here and hereafter,
$|\textbf{F}|:=(\sum_{i=0}^n|u_i|^2)^{1/2}$.
Moreover, it holds true that
$$\|u^\ast\|_{\vlp}\sim \sup_{t\in(0,\fz)}\||\textbf{F}(\cdot,t)|\|_{\vlp}$$
with the implicit equivalent positive constants independent of $u$ and $\textbf{F}$.
\end{proposition}

To prove Proposition \ref{p-hv1}, we first establish a technical lemma with respect to
the least
harmonic majorant of the harmonic vector considered in Proposition \ref{p-hv1}
(see \cite[p.\,80]{StWe71} for the definition of the least harmonic majorant).

\begin{lemma}\label{l-hv}
Let $p(\cdot)\in \cm\cp(\rn)$ satisfy $p_-\in(\frac{n-1}n,\fz)$.
If $\textbf{F}$ satisfies \eqref{gcr-equ} and \eqref{hbf-norm},
then, for any $\eta\in[\frac {n-1}n,p_-)$,
$a\in(0,\fz)$ and $(x,t)\in\urn$,
\begin{equation}\label{hv-esti}
|\textbf{F}(x,t+a)|\le \lf\{\lf(|\textbf{F}(\cdot,a)|^\eta\ast P_t\r)(x)\r\}^{1/\eta},
\end{equation}
where $P_t$ denotes the Poisson kernel as in \eqref{po-ke}.
Moreover, there exists a measurable function $g$ such that
$g(x)=\lim_{a\to0^+}|\textbf{F}(x,a)|$
pointwise almost every $x\in\rn$
and
\begin{equation}\label{hv-esti0}
|\textbf{F}(x,t)|\le \lf\{\lf(g^\eta\ast P_t\r)(x)\r\}^{1/\eta},
\end{equation}
here and hereafter, $\lim_{a\to 0^+}$ means $a\in(0,\fz)$ and $a\to 0$.
\end{lemma}

\begin{proof}
To prove this lemma, for all $x\in\rn$, $a\in[0,\fz)$ and $t\in(0,\fz)$, let
$\hbf_a(x,t):=\hbf(x,t+a)$ and,
for any $q\in(1,\frac{p_-}\eta)$, let
$$K(|\hbf_a(\cdot,t)|^{\eta q},t):=\int_\rn\frac{|\hbf_a(x,t)|^{\eta q}}{(|x|+1+t)^{n+1}}\,dx.$$

We claim that $K(|\hbf_a|^{\eta q},\cdot)$ is bounded on $(0,\fz)$.
To see this, for any $t\in(0,\fz)$, write
\begin{equation}\label{hv-esti1}
K(|\hbf_a(\cdot,t)|^{\eta q},t)=\int_{E_t}\frac{|\hbf(x,t+a)|^{\eta q}}{(|x|+1+t)^{n+1}}\,dx
+\int_{(E_t)^\complement}\cdots=:{\rm I}+{\rm II},
\end{equation}
where $E_t:=\{x\in\rn:\ |\hbf(x,t+a)|\ge1\}$.

Obviously, we have
\begin{eqnarray}\label{hv-esti2}
{\rm II}\le \int_{(E_t)^\complement}\frac1{(|x|+1+t)^{n+1}}\,dx
\ls\frac1{1+t}.
\end{eqnarray}
For I, by Remark \ref{r-vlp}(ii) and the fact that $\eta q<p_-$,
we find that
\begin{eqnarray*}
{\rm I}
&&\le \frac1{(1+t)^{n+1}}\int_{E_t}|\hbf(x,t+a)|^{p(x)}\,dx\\
&&\le \frac1{(1+t)^{n+1}}\int_{\rn}\lf[\frac{|\hbf(x,t+a)|}
{\||\hbf(\cdot,t+a)|\|_{\vlp}}\r]^{p(x)}
\||\hbf(\cdot,t+a)|\|_{\vlp}^{p(x)}\,dx\\
&&\le \frac1{(1+t)^{n+1}}\max\lf\{\||\hbf(\cdot,t+a)|\|_{\vlp}^{p_-},\,
\||\hbf(\cdot,t+a)|\|_{\vlp}^{p_+}\r\}\\
&&\le \frac1{(1+t)^{n+1}}\max\lf\{\sup_{t\in(0,\fz)}
\|\hbf(\cdot,t)\|_{\vlp}^{p_-},\,
\sup_{t\in(0,\fz)}\|\hbf(\cdot,t)\|_{\vlp}^{p_+}\r\}.
\end{eqnarray*}
This, together with \eqref{hv-esti1} and \eqref{hv-esti2}, implies that,
for all $t\in(0,\fz)$,
\begin{equation}\label{2.11x}
K(|\hbf_a(\cdot,t)|^{\eta q},t)\ls1
\end{equation} with the implicit positive constant
independent of $t$.
Therefore, the above claim holds true.

Since $\hbf$ satisfy \eqref{gcr-equ},
it follows, from \cite[p.\,234, Theorem 4.14]{StWe71},
that, for any $\eta\in[\frac{n-1}n,p_-)$, $|\hbf_a|^\eta$ is subharmonic
(see \cite[p.\,76]{StWe71} for the definition of subharmonic),
with $a\in[0,\fz)$.
From this, the above claim and \cite[Theorem 8]{Nua7312}, we further deduce that
the least harmonic majorant of $|\hbf_a|^\eta$ in $\urn$
exists and is given by
$g^\eta\ast P_t$,
where $g$ is a measurable function and
$g(x)=\lim_{t\to0^+}|\hbf_a(x,t)|$ pointwise almost every $x\in\rn$.
Thus, by taking $a=0$, we then obtain \eqref{hv-esti0}; moreover, 
if $a\in(0,\fz)$, this, combined with $g=|\hbf_a(\cdot,0)|=|\hbf(\cdot,a)|$,
implies that \eqref{hv-esti} holds true.
This finishes the proof of Lemma \ref{l-hv}.
\end{proof}

\begin{remark}\label{r-hv}
(i) We point out that, in the proof of Lemma \ref{l-hv},
the condition $p_-\in(\frac{n-1}n,\fz)$ is required merely because that
we need the fact that $|\hbf|^\eta$ is subharmonic on $\urn$ for any
$\eta\in [\frac{n-1}n,p_-)$. Thus, if there exists some $\eta_0\in(0,\frac{n-1}n)$
such that $|\hbf|^{\eta_0}$ is subharmonic on $\urn$, then, for all
$p_-\in(\eta_0,\fz)$, the conclusion of Lemma \ref{l-hv} still holds true.
Moreover, if, for all $\eta\in(0,\fz)$, $|\hbf|^{\eta}$ is subharmonic on $\urn$,
then, for every $p(\cdot)\in \cm\cp(\rn)$, by taking $\eta$ sufficiently
small, we always have $p_->\eta$ and, therefore,
Lemma \ref{l-hv} holds true for every
$p(\cdot)\in \cm\cp(\rn)$.

(ii) In Lemma~\ref{l-hv},
if $\hbf$ satisfies \eqref{gcr-equ} and
\begin{equation}\label{hbf-K}
\sup_{t\in(0,\fz)}
\int_\rn\frac{|\hbf(x,a+t)|^{\eta q}}{(|x|+1+t)^{n+1}}\,dx<\fz,
\quad q\in(1,p_-/\eta)
\end{equation}
for each $a\in[0,\fz)$,
instead of \eqref{hbf-norm},
then we have the same conclusion.
For example, the condition that
\begin{equation*}
 \sup_{t\in(0,\fz)}\||\hbf(\cdot,t)|\|_{L^r(\rn)}<\fz
 \quad (r\in[p_-,\fz))
\end{equation*}
implies \eqref{hbf-K}.

(iii) When $p(\cdot)\equiv p\in(\frac{n-1}n,\fz)$ is a constant, Lemma \ref{l-hv}
was proved in \cite[p.\,122]{stein93} (see also \cite[Lemma 2.12]{ccyy15}).
\end{remark}
To prove Proposition \ref{p-hv1}, we also need a density result as follows.

\begin{lemma}\label{l-dense}
Let $p(\cdot)\in\cm\cp(\rn)$. Then

{\rm(i)} for all $f\in\vhs$, $f$ is restricted at infinity, namely,
for any $r\in[p_+,\fz)$, it holds true that
$f\ast \Phi\in L^r(\rn)$ for all $\Phi\in\cs(\rn)$;

{\rm(ii)} for all $r\in[p_+,\fz)$, $\vhs\cap L^{1+r}(\rn)$
is dense in $\vhs$. Particularly, if $p_+\in(0,1]$, then $\vhs\cap L^2(\rn)$
is dense in $\vhs$.
\end{lemma}

\begin{proof}
We first prove (i).
We may assume that $f\ne0$.
By the definition of the non-tangential maximal function in \eqref{ntmf},
we see that, for all $x\in\rn$
and $y\in B(x,1)$,
\begin{equation}\label{rtc-1}
|(f\ast \Phi)(x)|\le M_{\Phi,\triangledown}^\ast(f)(y).
\end{equation}
If $x\in\rn$ satisfies that
$|(f\ast \Phi)(x)|>\|M_{\Phi,\triangledown}^\ast(f)\|_{\vlp}$,
then, by integrating on \eqref{rtc-1} over $B(x,1)$,
Remark \ref{r-vlp}(ii) and Lemma \ref{l-vhs1},
we find that
\begin{eqnarray*}
|f\ast \Phi(x)|^{p_-}
&&\le
\frac1{|B(x,1)|}
\int_{B(x,1)}
 \lf[
 \frac{M_{\Phi,\triangledown}^\ast(f)(y)}
     {\lf\|M_{\Phi,\triangledown}^\ast(f)\r\|_{\vlp}}
 \r]^{p(y)}
 \lf\|M_{\Phi,\triangledown}^\ast(f)\r\|_{\vlp}^{p_-}
\,dy\\
&&\ls
\lf\|M_{\Phi,\triangledown}^\ast(f)\r\|_{\vlp}^{p_-}
\ls \|f\|_{\vhs}^{p_-}<\fz.
\end{eqnarray*}
Thus, $f\ast \Phi\in L^\fz(\rn)$ and
\begin{eqnarray}\label{rtc-2}
\|f\ast\Phi\|_{L^\fz(\rn)}\ls \|f\|_{\vhs}.
\end{eqnarray}
Moreover, from \eqref{rtc-1} and Lemma \ref{l-vhs1},
we deduce that
\begin{eqnarray}\label{rtc-1x}
\|f\ast \Phi\|_{\vlp}\le\|M_{\Phi,\triangledown}^\ast(f)\|_{\vlp}\ls\|f\|_{\vhs}.
\end{eqnarray}
Let
$$g_1:=f\ast\Phi\chi_{\{x\in\rn:\ |f\ast\Phi(x)|\le1\}}\quad{\rm and}\quad
g_2:=f\ast\Phi\chi_{\{x\in\rn:\ |f\ast\Phi(x)|>1\}}.$$
Then $f\ast \Phi=g_1+g_2$ and, from \eqref{rtc-1x}, we easily deduce
that $g_1\in L^{p_+}(\rn)$ and $g_2\in L^{p_-}(\rn)$.
By this and \eqref{rtc-2}, we further conclude that, for any $r\in[p_+,\fz)$,
$f\ast\Phi\in L^r(\rn)$ and hence $f$ is a distribution restricted at infinity.

Next, we show (ii). Let $f\in \vhs$. Then $f$ is a bounded distribution due to
Lemma \ref{l-vhs} and hence, for any $t\in(0,\fz)$,
$f\ast P_t$ is well defined by Remark \ref{r-vhs}.
By the proof of \cite[Proposition 4.2]{cw14}, we know that
$\{f\ast P_t\}_{t\in(0,\fz)}$ converges to $f$ in $\vhs$.
Thus, to prove (ii), we only need to show that, for any $t\in(0,\fz)$
and $r\in[p_+,\fz)$, $f\ast P_t\in L^{1+r}(\rn)$. To this end,
let $\psi^{(1)}$, $\psi^{(2)}\in\cs(\rn)$ and $h\in L^1(\rn)$
be as in Remark \ref{r-welld} such that \eqref{poisson}
holds true. Then, for all $t\in(0,\fz)$,
$f\ast P_t=f\ast (\psi^{(1)})_t\ast h_t+f\ast(\psi^{(2)})_t.$
By (i), we see that, for any $t\in(0,\fz)$ and $r\in [p_+,\fz)$,
$f\ast (\psi^{(l)})_t\in L^{1+r}(\rn)$ for $l\in\{1,2\}$,
which, together with the Minkowski inequality for
both norms and integrals, implies that
\begin{eqnarray*}
\|f\ast P_t\|_{L^{1+r}(\rn)}
&&\le\lf\|f\ast\lf(\psi^{(1)}\r)_t\ast h_t\r\|_{L^{1+r}(\rn)}+
\lf\|f\ast\lf(\psi^{(2)}\r)_t\r\|_{L^{1+r}(\rn)}\\
&&\le\lf\|f\ast\lf(\psi^{(1)}\r)_t\r\|_{L^{1+r}(\rn)}\|h_t\|_{L^1(\rn)}
+\lf\|f\ast\lf(\psi^{(2)}\r)_t\r\|_{L^{1+r}(\rn)}<\fz.
\end{eqnarray*}
Therefore, for any $t\in(0,\fz)$, $f\ast P_t\in L^{1+r}(\rn)$, which completes
the proof of (ii) and hence Lemma \ref{l-dense}.
\end{proof}

\begin{remark}
(i) In \cite[Proposition 4.2]{cw14}, Cruz-Uribe and Wang
proved that $\vhs\cap L_\loc^1(\rn)$ is dense in $\vhs$.

(ii) We point out that, under the assumption that
$p(\cdot)$ satisfies the globally H\"older continuous condition as in
Remark \ref{r-defn}, by using the Calder\'on-Zygmund decomposition,
Nakai and Sawano \cite[p.\,3693]{ns12} also showed that
$\vhs\cap L^{1+p_+}(\rn)$ is dense in $\vhs$.
\end{remark}

In \cite[Theorem 8.4]{cw14}, Cruz-Uribe and Wang studied the boundedness
of the convolution type Calder\'on-Zygmund singular integrals with sufficient
regularity on $\vhs$. As a special case of \cite[Theorem 8.4]{cw14},
we immediately obtain the following conclusion, the details being omitted.

\begin{lemma}\label{l-rtb}
Let $p(\cdot)\in \cm\cp(\rn)$. Then, for any $j\in\{1,\dots,n\}$,
the Riesz transform $R_j$ is bounded on $\vhs$.
\end{lemma}

Now we are ready to prove Proposition \ref{p-hv1}.

\begin{proof}[Proof of Proposition \ref{p-hv1}]
We first prove the sufficiency. Let $u$ be a harmonic function on $\urn$.
Assume that $\hbf :=\{u_0,u_1,\dots,u_n\}$ satisfies \eqref{gcr-equ},
\eqref{hbf-norm} and $u_0:=u$.
Since, for any $\eta\in [\frac{n-1}n, p_0)$
and all $t\in(0,\fz)$,
\begin{equation*}
\lf\||\hbf(\cdot,t)|^\eta\r\|_{L^{\frac{p(\cdot)}\eta}(\rn)}^{1/\eta}
=\lf\||\hbf(\cdot,t)|\r\|_{\vlp}
\le \sup_{t\in(0,\fz)}\lf\||\hbf(\cdot,t)|\r\|_{\vlp}<\fz,
\end{equation*}
it follows that $\{|\hbf(\cdot,t)|^\eta\}_{t\in(0,\fz)}$ is a bounded set of
$L^{\frac {p(\cdot)}\eta}(\rn)$. From this, the reflexivity of the Banach space
$L^{\frac {p(\cdot)}\eta}(\rn)$ (see \cite[Corollary 2.81]{cfbook}
and also \cite[Theorem 3.4.7]{dhr11}) and the weak compactness of reflexive
Banach spaces (see \cite[p.\,126, Theorem 1]{Yo78}), we deduce that there exist a
subsequence $\{|\hbf(\cdot,t_k)|^\eta\}_{k\in\nn}$
and an element $h\in L^{\frac {p(\cdot)}\eta}(\rn)$ such that
$t_k\to 0^+$ as $k\to\fz$ and $\{|\hbf(\cdot,t_k)|^\eta\}_{k\in\nn}$
converges weakly to $h$ in $L^{\frac {p(\cdot)}\eta}(\rn)$. Thus,
for any $g\in L^{(\frac {p(\cdot)}\eta)^\ast}(\rn)$, we have
\begin{equation}\label{hv1-esti}
\lim_{k\to\fz}\int_\rn |\hbf(y,t_k)|^\eta g(y)\,dy
=\int_\rn h(y)g(y)\,dy,
\end{equation}
where $(\frac{p(\cdot)}{\eta})^\ast$ denotes the dual variable exponent
of $\frac{p(\cdot)}{\eta}$. From this, it is easy to deduce that $h$ is non-negative
almost everywhere in $\rn$.
For any given $(x,t)\in\urn$, we let $g_0(y):=P_t(x-y)$ for all $y\in\rn$.
Then $g_0\in L^{(\frac {p(\cdot)}\eta)^\ast}(\rn)$ and, by \eqref{hv1-esti}, we obtain
\begin{equation*}
\lim_{k\to\fz}|\hbf(\cdot,t_k)|^\eta\ast P_t(x)=h\ast P_t(x),
\end{equation*}
which, combined with the continuity of $|\hbf|$
and \eqref{hv-esti}, implies that, for any given $(x,t)\in\urn$,
\begin{eqnarray*}
|\hbf(x,t)|^\eta=\lim_{k\to\fz}|\hbf(x,t_k+t)|^\eta
\le \lim_{k\to\fz}|\hbf(x,t_k)|^\eta\ast P_t(x)=h\ast P_t(x).
\end{eqnarray*}
By this, the fact that $h\ast P_t(x)\ls \cm(h)(x)$ for all $(x,t)\in\urn$,
$p(\cdot)\in\cm\cp(\rn)$ together with Remark \ref{r-defn}(iii), (iv) and (v) of Remark \ref{r-vlp},
we conclude that
\begin{eqnarray}\label{2.17x}
\|u^\ast\|_{\vlp}
&&\le \lf\|(|\hbf|^\eta)^\ast\r\|_{L^{\frac {p(\cdot)}\eta}(\rn)}^{1/\eta}
\ls \lf\|\cm(h)\r\|_{L^{\frac {p(\cdot)}\eta}(\rn)}^{1/\eta}
\ls \lf\|h\r\|_{L^{\frac {p(\cdot)}\eta}(\rn)}^{1/\eta}\\
&&\sim \sup_{\|g\|_{L^{({p(\cdot)}/\eta)^\ast}(\rn)}\le 1}
\lf|\int_\rn h(y)g(y)\,dy\r|^{1/\eta}\noz\\
&&\sim \sup_{\|g\|_{L^{({p(\cdot)}/\eta)^\ast}(\rn)}\le 1}
\lim_{k\to\fz}\lf|\int_\rn |\hbf(y,t_k)|^\eta g(y)\,dy\r|^{1/\eta}\noz\\
&&\ls \lim_{k\to\fz}\||\hbf(\cdot,t_k)|^\eta
\|_{L^{\frac{p(\cdot)}{\eta}}(\rn)}^{1/\eta}
\ls\sup_{t\in(0,\fz)}\||\hbf(\cdot,t)|\|_{\vlp}<\fz.\noz
\end{eqnarray}
This implies that $u^\ast\in \vlp$ and finishes the proof of
sufficiency.

Conversely, we show the necessity. Let $u$ be a harmonic function on $\urn$
and $u^\ast\in\vlp$.
Then, by Proposition \ref{p-hc}, there exists $f\in \vhs$ such that
$u(x,t)=f\ast P_t(x)$ for all $(x,t)\in\urn$
and
\begin{equation}\label{hv1-esti0}
\|f\|_{\vhs}\ls\|u^\ast\|_{\vlp}.
\end{equation}
Since $\vhs\cap L^{1+r}(\rn)$ with $r\in[p_+,\fz)$ is dense in $\vhs$
(see Lemma \ref{l-dense}), it follows that there exist
$\{f_k\}_{k\in\nn}\st(\vhs\cap L^{1+r}(\rn))$ such that $f_k\to f$ in
$\vhs$ as $k\to\fz$ and
\begin{equation}\label{hv1-esti1}
\|f_k\|_{\vhs}\ls \|f\|_{\vhs}.
\end{equation}

For each $k\in\nn$, $j\in\{1,\,\dots,\,n\}$ and $(x,t)\in\urn$, let
$$u^{(0)}_k(x,t):=(f_k\ast P_t)(x)\quad
{\rm and}\quad u^{(j)}_k(x,t):=(f_k\ast Q_t^{(j)})(x),$$
where $P_t$ is the Poisson kernel as in \eqref{po-ke} and $Q_t^{(j)}$ the
\emph{$j$-th conjugate Poisson kernel} defined by setting, for all $x\in\rn$,
$$Q_t^{(j)}(x):=\frac{\Gamma([n+1]/2)}{\pi^{(n+1)/2}}\frac {x_j}{(t^2+|x|^2)^{(n+1)/2}}.$$
From \cite[p.\,236, Theorem 3.17]{StWe71} and the fact that
$\{f_k\}_{k\in\nn}\st L^{1+r}(\rn)$, we deduce that the harmonic vector
$\hbf_k:=\{u_k^{(0)},u_k^{(1)},\dots,u_k^{(n)}\}$ satisfies the generalized Cauchy-Riemann equation
\eqref{gcr-equ}. Moreover, \cite[p.\,65, Theorem 3 and p.\,78, Item 4.4]{stein70}
implies that, for each $j\in\{1,\dots,n\}$ and all
$(x,t)\in\urn$, $(f_k\ast Q_t^{(j)})(x)=(R_j(f_k)\ast P_t)(x)$,
due to the fact that
$\{f_k\}_{k\in\nn}\st L^{1+r}(\rn)$.
From this, Lemmas \ref{l-vhs} and \ref{l-rtb}, and \eqref{hv1-esti1}, we deduce that,
for any $k\in\nn$ and $j\in\{1,\dots,n\}$,
\begin{eqnarray}\label{hv1-esti2}
\sup_{t\in(0,\fz)}\lf\|u_k^{(j)}(\cdot,t)\r\|_{\vlp}\ls\|R_j(f_k)\|_{\vhs}
\ls\|f_k\|_{\vhs}\ls\|f\|_{\vhs}
\end{eqnarray}
and, similarly, we also have
\begin{equation*}
\sup_{t\in(0,\fz)}\lf\|u_k^{(0)}(\cdot,t)\r\|_{\vlp}
\ls\|f_k\|_{\vhs}\ls\|f\|_{\vhs}.
\end{equation*}
By this, \eqref{hv1-esti0} and \eqref{hv1-esti2}, we find that
\begin{equation}\label{hv1-esti3-x}
\sup_{t\in(0,\fz)}\||\hbf_k(\cdot,t)|\|_{\vlp}\ls \|f\|_{\vhs}
\ls\|u^\ast\|_{\vlp}.
\end{equation}

Next, we claim that there exists $\{k_i\}_{i\in\nn}\st\nn$ such that
$k_i\to\fz$ as $i\to\fz$ and, for all $(x,t)\in\urn$,
\begin{equation}\label{hv1-esti3}
\lim_{i\to\fz}f_{k_i}\ast P_t(x)=f\ast P_t(x),\quad
\lim_{i\to\fz}R_j(f_{k_i})\ast P_t(x)=R_j(f)\ast P_t(x).
\end{equation}
Indeed, since $\{f_k\}_{k\in\nn}$ converges to $f$ in $\vhs$,
it follows, from Lemma \ref{l-vhs}, that
\begin{equation*}
\lim_{k\to\fz}\|f_k\ast P_t-f\ast P_t\|_{\vlp}\ls\lim_{k\to\fz}\|f_k-f\|_{\vhs}=0
\end{equation*}
and, from Lemmas \ref{l-vhs} and \ref{l-rtb}, that,
for any $j\in\{1,\dots,n\}$,
\begin{eqnarray*}
&&\lim_{k\to\fz}\|R_j(f_k)\ast P_t-R_j(f)\ast P_t\|_{\vlp}\\
&&\hs\ls\lim_{k\to\fz}\|R_j(f_k-f)\|_{\vhs}
\ls\lim_{k\to\fz}\|f_k-f\|_{\vhs}=0.
\end{eqnarray*}
By this, \cite[Proposition 2.67]{cfbook} and Remark \ref{r-welld}, we further
conclude that there exists $\{k_i\}_{i\in\nn}\st\nn$ such that
$k_i\to\fz$ as $i\to\fz$ and
\eqref{hv1-esti3} holds true for all $(x,t)\in\urn$. Therefore, the above claim holds
true.

On the other hand, by Lemma \ref{l-rtb}, we know that $\{R_j(f)\}_{j=1}^n\st\vhs$,
which, together with Lemma \ref{l-vhs}, implies that $\{R_j(f)\}_{j=1}^n$
are bounded distributions.
Thus, $\{R_j(f)\ast P_t\}_{j=1}^n$ are harmonic due
to Remark \ref{r-welld} and hence
$\hbf:=\{f\ast P_t,R_1(f)\ast P_t,\dots,R_n(f)\ast P_t\}$ satisfies the generalized
Cauchy-Riemann equation \eqref{gcr-equ} (see \cite[p.\,122]{stein93}).

Finally, by the above claim, the Fatou lemma (see \cite[Theorem 2.61]{cfbook})
and \eqref{hv1-esti3-x}, we conclude that
\begin{eqnarray*}
\sup_{t\in(0,\fz)}\||\hbf(\cdot,t)|\|_{\vlp}
&&=\sup_{t\in(0,\fz)}\lf\|\lim_{i\to\fz}|\hbf_{k_i}(\cdot,t)|\r\|_{\vlp}\\
&&\le \sup_{t\in(0,\fz)}\liminf_{i\to\fz}\||\hbf_{k_i}(\cdot,t)|\|_{\vlp}
\ls \|u^\ast\|_{\vlp}.
\end{eqnarray*}
This finishes the proof of Proposition \ref{p-hv1}.
\end{proof}

The following lemma can be proved by an argument similar to that used in the
proof of \cite[Theorem~3.5]{ns14} (see also \cite[Corollary 4.20]{zsy15}),
the details being omitted.
\begin{lemma}\label{l-p}
Let $p(\cdot)\in \cm\cp(\rn)$ and $p_-\in[1,\fz)$.
Then $\vhs\subset\vlp$ and there exists a positive constant $C$ such that,
for all $f\in\vhs$,
$$\|f\|_{\vlp}\le C\|f\|_{\vhs}.$$
\end{lemma}

We now turn to prove Theorem \ref{t-rtc1}.

\begin{proof}[Proof of Theorem \ref{t-rtc1}]
We first prove ``(i)$\Rightarrow$(ii)". Let $f\in\vhs$. Then, by Lemma \ref{l-dense},
we see that $f$ is a distribution restricted at infinity.
On the other hand, by the definition of the radial grand maximal function in
\eqref{granmf}
and Lemma \ref{l-rtb}, we easily find that, for all $\ez\in(0,\fz)$,
\begin{eqnarray*}
&&\|f\ast \phi_\ez\|_{\vlp}+\sum_{j=1}^n\|R_j(f)\ast\phi_\ez\|_{\vlp}\\
&&\hs\ls\|f_{N,+}^\ast\|_{\vlp}+\sum_{j=1}^n\|(R_j(f))_{N,+}^\ast\|_{\vlp}\\
&&\hs\sim \|f\|_{\vhs}+\sum_{j=1}^n\|R_j(f)\|_{\vhs}
\ls\|f\|_{\vhs}.
\end{eqnarray*}
Furthermore, if $p_-\in[1,\fz)$, then, using Lemmas~\ref{l-rtb}
and \ref{l-p},
we have
$$
 \|f\|_{\vlp}+\sum_{j=1}^n\|R_j(f)\|_{\vlp}
 \ls
 \|f\|_{\vhs}+\sum_{j=1}^n\|R_j(f)\|_{\vhs}
 \ls\|f\|_{\vhs},
$$
which completes the proof of ``(i)$\Rightarrow$(ii)".

Next, we show ``(ii)$\Rightarrow$(i)".
Suppose that $f$ is a distribution restricted at infinity and \eqref{rtc-3} holds true.
For any $\ez\in(0,\fz)$, let
$\hbf_\ez:=\{u_0^{(\ez)},u_1^{(\ez)},\dots,u_n^{(\ez)}\},$
where $u_0^{(\ez)}:=f_\ez\ast P_t$, $u_j^{(\ez)}:=f_\ez\ast Q_t^{(j)}$ for each $j\in\{1,\dots,n\}$
and $f_\ez:=f\ast \phi_\ez$. Since $f$ is a distribution restricted at infinity,
it follows that $f\ast \phi_\ez\in L^r(\rn)$
with $r\in[1+p_+,\fz)$ and hence
$\{u_j^{(\ez)}\}_{j=0}^n$ are harmonic on $\urn$, continuous
and bounded on its closure (see \cite[pp.\,123-124]{stein93}). Moreover,
by \cite[p.\,236, Theorem 4.17]{StWe71}, we know that
$\hbf_\ez$ satisfies the generalized Cauchy-Riemann equation \eqref{gcr-equ}
and $\sup_{t\in(0,\fz)}\||\hbf_\ez(\cdot,t)|\|_{L^r(\rn)}<\fz$.
Thus, by Lemma \ref{l-hv} and Remark \ref{r-hv}(ii), we see that, for all
$\eta\in[\frac{n-1}n,p_0)$ and $(x,t)\in\urn$,
\begin{equation*}
|\hbf_\ez(x,t)|^\eta\le(|\hbf_\ez(x,0)|^\eta\ast P_t)(x),
\end{equation*}
where $\hbf_\ez(\cdot,0)
:=\{f\ast\phi_\ez, R_1(f)\ast\phi_\ez,\dots,R_n(f)\ast \phi_\ez\}$.
By this and Remark \ref{r-defn}(iii), we conclude that,
for all $\vez,\,t\in(0,\fz)$,
\begin{eqnarray*}
\lf\|\hbf_\ez(\cdot,t)\r\|_{\vlp}
&&\le\lf\||\hbf_\ez(\cdot,0)|^\eta\ast P_t\r\|_{L^{\frac{p(\cdot)}\eta}(\rn)}^{1/\eta}
\ls \lf\|\cm(|\hbf_\ez(\cdot,0)|^\eta)\r\|_{L^{\frac{p(\cdot)}\eta}(\rn)}^{1/\eta}\\
&&\ls \||\hbf_\ez(\cdot,0)|^\eta\|_{L^{\frac{p(\cdot)}\eta}(\rn)}^{1/\eta}
\sim\||\hbf_\ez(\cdot,0)|\|_{\vlp}\\
&&\ls\|f\ast\phi_\ez\|_{\vlp}+\sum_{j=1}^n\|R_j(f)\ast\phi_\ez\|_{\vlp}
\ls A_3,
\end{eqnarray*}
which, implies that
$\sup_{t\in(0,\fz)}\|\hbf_\ez(\cdot,t)\|_{\vlp}\ls A_3<\fz.$
Thus, by Proposition \ref{p-hv1}, we know that
$(f_\ez)_P^\ast\in\vlp$ and $\|(f_\ez)_P^\ast\|_\vlp\ls A_3$.
Observe that $f\ast P_t\in L^{1+r}(\rn)$ when $r\in [p_+,\fz)$
(see the proof of Lemma \ref{l-dense}(ii)), it follows,
from \cite[p.\,10, Theorem 1.18]{StWe71}, that $f\ast P_t\ast\phi_{\ez}$
converges in measure to $f\ast P_t$ as $\ez\to\fz$, which, together with,
the Riesz lemma (see, for example, \cite[Theorem 1.1.13]{Gra1}) and
Remark \ref{r-welld}, implies that
there exist $\{\ez_k\}_{k\in\nn}\st(0,\fz)$ such that $\ez_k\to 0^+$ as
$k\to\fz$,
for all $(x,t)\in\urn$,
$\lim_{k\to\fz} f\ast P_t\ast\phi_{\ez_k}=f\ast P_t$.
Therefore, for all $x\in \rn$
\begin{eqnarray*}
f_P^\ast(x)
&&=\sup_{|y-x|<t}|f\ast P_t(y)|=\sup_{|y-x|<t}\lim_{k\to\fz}|f_{\ez_k}\ast P_t(y)|\\
&&\le \lim_{k\to\fz}\sup_{|y-x|<t}|f_{\ez_k}\ast P_t(y)|=\lim_{k\to\fz}(f_{\ez_k})_P^\ast(x),
\end{eqnarray*}
which,
combined with Lemma \ref{l-vhs} and the Fatou lemma (see \cite[Theorem 2.61]{cfbook}),
implies that $f\in\vhs$ and
$$\|f\|_{\vhs}\sim\|f_P^\ast \|_{\vlp}
\ls\lf\|\lim_{k\to\fz}(f_{\ez_k})_P^\ast\r\|_{\vlp}\ls A_3.$$
This finishes the proof of ``(ii)$\Rightarrow$(i)".

Furthermore,
if $p_-\in[1,\fz)$, then \eqref{rtc-3a} implies that
$$\{f,R_1(f),\dots,R_n(f)\}\st\vlp\subset L^{p_-}(\rn)+L^{p_+}(\rn).$$
Let $f=f_{0,1}+f_{0,2}$ and $R_j(f)=f_{j,1}+f_{j,2}$ for all $j\in\{1,\dots,n\}$,
with $f_{j,1}\in L^{p_-}(\rn)$ and $f_{j,2}\in L^{p_+}(\rn)$ for all
$j\in\{0,1,\dots,n\}$.
Then, for all $(x,t)\in\urn$,
\begin{eqnarray*}
\hbf(x,t):=&&\{f\ast P_t(x), R_1(f)\ast P_t(x),\,\dots,\,R_n(f)\ast P_t(x)\}\\
=&&\{f_{0,1}\ast P_t(x)+f_{0,2}\ast P_t(x),\,\dots,\,f_{n,1}\ast P_t(x)+f_{n,2}\ast P_t(x)\}
\end{eqnarray*}
and hence
$$
|\hbf(x,t)|\ls\sum_{j=0}^n\sum_{i=1}^2|f_{j,i}*P_t(x)|,
\quad
\sup_{t\in(0,\fz)}\|f_{j,1}*P_t\|_{L^{p_-}(\rn)}+
\sup_{t\in(0,\fz)}\|f_{j,2}*P_t\|_{L^{p_+}(\rn)}<\fz,
$$
and
$$
 \lim_{t\to0^+} f*P_t=f, \quad
 \lim_{t\to0^+} R_j(f)*P_t=R_j(f), \ j\in\{1,\dots,n\},
$$
pointwise almost everywhere (see \cite[Theorem~5.8]{cfbook}).
Thus, by an argument similar to that used in the proof of
\eqref{2.11x}, we conclude that \eqref{hbf-K} holds true for
$\eta\in[\frac{n-1}n,p_0)$ and hence, by Remark \ref{r-hv}(ii), we know that,
for all $(x,t)\in\urn$,
\begin{equation*}
|\hbf(x,t)|^\eta\le(|\hbf(\cdot,0)|^\eta\ast P_t)(x),
\end{equation*}
where $\hbf(\cdot,0):=\{f, R_1(f),\dots,R_n(f)\}$.
From this, $p(\cdot)\in\cm\cp(\rn)$ and Remark \ref{r-defn}(iii),
we further deduce that
\begin{eqnarray*}
\sup_{t\in(0,\fz)}\|\hbf(\cdot,t)\|_{\vlp}
&&\ls  \sup_{t\in(0,\fz)}\lf\|\lf\{|\hbf(\cdot,0)|^\eta\ast P_t\r\}^{1/\eta}\r\|_\vlp\\
&&\ls \lf\|\cm\lf(|\hbf(\cdot,0)|^\eta\r)\r\|_{L^{p(\cdot)/\eta}}^{1/\eta}
\ls \lf\||\hbf(\cdot,0)|^\eta\r\|_{L^{p(\cdot)/\eta}}^{1/\eta}
\sim\||\hbf(\cdot,0)|\|_\vlp\\
&&\ls \|f\|_{\vlp}+\sum_{j=1}^n\|R_j(f)\|_\vlp\ls
A_3.
\end{eqnarray*}
Therefore, by Lemma \ref{l-vhs} and Proposition \ref{p-hv1}, we have
$$\|f\|_{\vhs}\sim\|f_P^\ast \|_{\vlp}
\sim \sup_{t\in(0,\fz)}\|\hbf(\cdot,t)\|_{\vlp}\ls A_3.$$
This finishes the proof of Theorem \ref{t-rtc1}.
\end{proof}

\subsection{The case $p_-\in(0,\frac{n-1}n]$\label{s2.2}}
\hskip\parindent
In this subsection, we prove Theorem \ref{t-rtc2} and begin with some notions.
Following \cite[p.\,133]{stein93}, let $m\in\nn$ and $\{e_0,e_1,\dots,e_n\}$ be
an orthonormal basis of $\rr^{n+1}$.
Then the \emph{tensor product of $m$ copies of $\rr^{n+1}$} is defined to be the set
\begin{equation*}
\bigotimes^m\rr^{n+1}
:=\lf\{\emph{\textbf{G}}:=\sum_{j_1,\,\dots,\,j_m=0}^n
G_{j_1,\,\dots,\,j_m}e_{j_1}\otimes\cdots\otimes e_{j_m}:\
G_{j_1,\,\dots,\,j_m}\in\cc\r\},
\end{equation*}
where $e_{j_1}\otimes\cdots\otimes e_{j_m}$ denotes the \emph{tensor product}
of $e_{j_1},\,\dots,\,e_{j_m}$,
and each $\hbg\in \bigotimes\limits^m\rr^{n+1}$ is called a
\emph{tensor of rank $m$}.

Let $\hbg:\ \urn\to \bigotimes\limits^m\rr^{n+1}$
be a tensor-valued function of rank $m$ of the form that,
for all $(x,t)\in\urn$,
\begin{equation}\label{tensor-va}
\hbg(x,t)=\sum_{j_1,\,\dots,\,j_m=0}^n
G_{j_1,\,\dots,\,j_m}(x,t)e_{j_1}\otimes\cdots\otimes e_{j_m}
\end{equation}
with $G_{j_1,\,\dots,\,j_m}(x,t)\in\cc$.
Then the tensor-valued function $\hbg$ of rank $m$
is said to be \emph{symmetric}
if, for any permutation $\sz$ on $\{1,\dots,m\}$,
$j_1,\,\dots,\,j_m\in\{0,\,1,\,\dots,\,n\}$ and $(x,t)\in\rr^{n+1}$,
$$
 G_{j_1,\,\dots,\,j_m}(x,t)=G_{j_{\sz(1)},\,\dots,\,j_{\sz(m)}}(x,t).
$$
For $\hbg$ being symmetric, $\hbg$ is said to be of \emph{trace zero} if, for all
$j_3,\,\dots,\,j_m\in\{0,\,1,\,\dots,\,n\}$ and $(x,t)\in\rr^{n+1}$,
$$
 \sum_{j=0}^n G_{j,\,j,\,j_3,\,\dots,\,j_m}(x,t)\equiv 0.
$$
If $\hbg^{(1)},\ \hbg^{(2)}\in \bigotimes\limits^{m}\rr^{n+1}$ and
$$\hbg^{(l)}:=\sum_{j_1,\,\dots,\,j_m=0}^n
G_{j_1,\,\dots,\,j_m}^{(l)}e_{j_1}\otimes\cdots\otimes e_{j_m},\ l\in\{1,2\},$$
then define the \emph{product} of $\hbg^{(1)}$ and $\hbg^{(2)}$ by setting
$$
 \lf\langle \hbg^{(1)},\hbg^{(2)}\r\rangle:=\sum_{j_1,\,\dots,\,j_m=0}^n
 G_{j_1,\,\dots,\,j_m}^{(1)}G_{j_1,\,\dots,\,j_m}^{(2)}.
$$

Let $\hbg$ be as in \eqref{tensor-va}.
Its \emph{gradient}
$\nabla \hbg:\ \urn\to \bigotimes\limits^{m+1}\rr^{n+1}$
is a tensor-valued function of rank $m+1$
of the form that, for all $(x,t)\in\urn$,
\begin{eqnarray*}
\nabla \hbg(x,t)
&&=\sum_{j=0}^n\frac{\partial \hbg}{\partial x_j}(x,t)\otimes e_j\\
&&=\sum_{j=0}^n\sum_{j_1,\,\dots,\,j_m=0}^n\frac{\partial G_{j_1,\,\dots,\,j_m}}
{\partial x_j}(x,t)e_{j_1}\otimes\cdots\otimes e_{j_m}\otimes e_j,
\end{eqnarray*}
here and hereafter, we always let $x_0:=t$.
A tensor-valued function $\hbg$ is said to satisfy the
\emph{generalized Cauchy-Riemann equation} if both $\hbg$ and $\nabla \hbg$
are symmetric and of trace zero.
Obviously, if $m=1$, this definition of generalized Cauchy-Riemann equations is
equivalent to that as in \eqref{gcr-equ}. For more details on the
generalized Cauchy-Riemann equation on tensor-valued functions, we refer the reader
to \cite{StWe68}.

The following conclusion is just \cite[Theorem 1]{cz64}.

\begin{lemma}\label{l-ho2}
Let $m\in\nn$ and $u$ be a harmonic function on $\urn$. Then, for all
$\eta\in [\frac{n-1}{n+m-1},\fz)$, $|\nabla^mu|^\eta$ is subharmonic, where,
for all $(x,t)\in\urn$,
$$
 \nabla^mu(x,t):=
 \{\partial^\alpha u(x,t)\}_{\alpha\in\zz_+^{n+1},\,|\alpha|=m}
 $$
and, for all $\alpha:=\{\alpha_0,\,\dots,\,\alpha_n\}\in\zz_+^{n+1}$,
$|\alpha|:=\sum_{j=0}^n\alpha_j$, $x_0:=t$,
$\partial^\alpha:=(\frac{\partial}{\partial x_0})^{\alpha_0}
\cdots (\frac{\partial}{\partial x_n})^{\alpha_n}$ and
$$
 |\nabla^mu|:=\lf\{\sum_{\alpha\in\zz_+^{n+1},\,|\alpha|=m}
 |\partial^\alpha u(x,t)|^2\r\}^{\frac12}.
$$
\end{lemma}

The following lemma was proved in \cite[Theorem 14.3]{Uch01} (see also \cite{StWe71}).
\begin{lemma}\label{l-ho3}
Let $m\in\nn\cap[2,\fz)$, $\textbf{G}$ be a tensor-valued function of rank
$m$ satisfying that both $\textbf{G}$ and $\nabla \textbf{G}$ are symmetric,
and $\textbf{G}$ is of trace zero.
Then there exists a harmonic function $u$ on $\urn$ such that $\nabla^m u=\textbf{G}$,
namely, for all $j_1,\,\dots,\,j_m\in\{0,1,\dots,n\}$ and $(x,t)\in\urn$,
$$\frac{\partial}{\partial x_{j_1}}\cdots \frac{\partial}{\partial x_{j_m}}
u(x,t)=\textbf{G}_{j_1,\,\dots,\,j_m}(x,t).$$
\end{lemma}

From Lemmas \ref{l-ho2} and \ref{l-ho3}, we easily deduce
the following conclusion.
\begin{corollary}\label{c-ho1}
Let $m\in\nn\cap [2,\fz)$ and $\textbf{G}$ be a tensor-valued function of
rank $m$ satisfying that both $\textbf{G}$ and $\nabla \textbf{G}$ are symmetric,
and $\textbf{G}$ is of trace zero. Then, for all $\eta\in[\frac{n-1}{n+m-1},\fz)$,
$| \textbf{G}|^\eta$ is subharmonic on $\urn$.
\end{corollary}

\begin{proposition}\label{p-ho1}
Let $p(\cdot)\in\cm\cp(\rn)$, $m\in\nn$ and $p_-\in(\frac{n-1}{n+m-1},\fz)$.
Assume that $u$ is a harmonic function on $\urn$ and $\textbf{G}$ is
a tensor-valued function of rank $m$
satisfying generalized Cauchy-Riemann equation such that
$\langle \textbf{G},\otimes^me_0\rangle=u$ and
$$\sup_{t\in(0,\fz)}\|\textbf{G}(\cdot,t)\|_{\vlp}<\fz.$$
Then $u^\ast\in\vlp$ and there exists a positive constant $C$ such that
$$\|u^\ast\|_{\vlp}\le C\sup_{t\in(0,\fz)}\|\textbf{G}(\cdot,t)\|_{\vlp}.$$
\end{proposition}

\begin{proof}
Let $\eta\in[\frac{n-1}{n+m-1},p_-)$.
Observe that $\hbg$ satisfies the generalized Cauchy-Riemann equation,
it follows, from Corollary \ref{c-ho1}, that $|\hbg|^\eta$ is subharmonic.
Thus, by Lemma \ref{l-hv} and its proof,
we find that, for all $a\in(0,\fz)$ and $(x,t)\in\urn$,
$$|\hbg(x,t+a)|^{\eta}\le (|\hbg(\cdot,a)|^\eta\ast P_t)(x).$$
On the other hand, it is easy to see that
\begin{eqnarray*}
\sup_{t\in(0,\fz)}\||\hbg(\cdot,t)|^\eta\|_{L^{\frac{p(\cdot)}{\eta}}(\rn)}^{\frac1\eta}
= \sup_{t\in(0,\fz)}\||\hbg(\cdot,t)|\|_{\vlp}<\fz,
\end{eqnarray*}
namely, $\{|\hbg(\cdot,t)|^\eta\}_{t\in(0,\fz)}$ is bounded in
$L^{\frac{p(\cdot)}{\eta}}(\rn)$. Then, by an argument similar to that
used in the proof of Proposition \ref{p-hv1},
we conclude that there exist a subsequence $\{\hbg(\cdot,t_k)\}_{k\in\nn}$
and $h\in L^{\frac{p(\cdot)}{\eta}}(\rn)$ such that $h$ is non-negative almost
everywhere in $\rn$ and, for all $(x,t)\in\urn$,
\begin{eqnarray*}
|\hbg(x,t)|^\eta=\lim_{k\to\fz}|\hbg(x,t_k+t)|^\eta
\le \lim_{k\to\fz}|\hbg(x,t_k)|^\eta\ast P_t(x)=h\ast P_t(x).
\end{eqnarray*}
Moreover, by this, we see that, for all $x\in\rn$,
\begin{eqnarray*}
u^\ast(x)
&&=\sup_{|y-x|<t}|u(y,t)|
=\sup_{|y-x|<t}\lf|\lf\langle \hbg(y,t),\otimes^{m} e_0\r\rangle\r|\\
&&\le \sup_{|y-x|<t}|\hbg(y,t)|\ls \lf[\cm(h)(x)\r]^{\frac1\eta}
\end{eqnarray*}
and hence, by an argument similar to that used in the proof of \eqref{2.17x},
we find that
$$\|u^\ast\|_{\vlp}\ls \lf\|\lf[\cm(h)\r]^{\frac1\eta}\r\|_{\vlp}
\ls \sup_{t\in(0,\fz)}\||\hbg(\cdot,t)|\|_{\vlp}.$$
This finishes the proof of Proposition \ref{p-ho1}.
\end{proof}

We now prove Theorem \ref{t-rtc2}.

\begin{proof}[Proof of Theorem \ref{t-rtc2}]
To prove this theorem, it suffices to show ``(ii)$\Rightarrow$(i)",
since the proof of ``(i)$\Rightarrow$(ii)" is similar to that of Theorem \ref{t-rtc1}.

For all $\vez\in(0,\fz)$ and $(x,t)\in\urn$,
let $f_\vez(x):=f\ast \phi_\vez(x)$,
$$
 G_{j_1,\,\dots,\,j_m}^{(\vez)}(x,t)
 :=(R_{j_1}\cdots R_{j_m}(f_\vez)\ast P_t)(x)
$$
and
$$
 \hbg^{(\vez)}(x,t)
 :=\sum_{j_1,\,\dots,\,j_m=0}^\fz
  G_{j_1,\,\dots,\,j_m}^{(\vez)}(x,t) e_{j_1}\otimes\cdots\otimes e_{j_m},
$$
where $P_t$ denotes the Poisson kernel as in \eqref{po-ke}
and $R_0:=I$ denotes the identity operator.
Since $f\ast \phi_\vez\in L^{1+r}(\rn)$ with $r\in[p_+,\fz)$
due to Lemma \ref{l-dense}(i),
it follows, from the Fourier
transform, that both $\hbg^{(\vez)}$ and $\nabla \hbg^{(\vez)}$ are symmetric,
and $\hbg^{(\vez)}$ is of trace zero.
Then, by Proposition \ref{p-ho1}, we find that $(f_\vez)_P^\ast \in\vlp$
and
$\|(f_\vez)_P^\ast\|_{\vlp}
\ls A_4$.
On the other hand, by the proof of Lemma \ref{l-dense}(ii),
we see that $f\ast P_t\in L^{1+r}(\rn)$. From this, \cite[p.\,10, Theorem 1.18]{StWe71}
and the fact that $\int_\rn \phi(x)\,dx=1$, we deduce that
$f\ast P_t\ast \phi_\vez$ converges to $f\ast P_t$ in measure as
$\vez\to0$, which, combined with the Riesz lemma
(see, for example, \cite[Theorem 1.1.13]{Gra1}),
implies that, there exist $\{\vez_k\}_{k\in\nn}\st(0,\fz)$ such that $\vez_k\to 0^+$ as
$k\to\fz$ and, for any given $(x,t)\in \urn$,
$$\lim_{k\to\fz} f_{\vez_k}\ast P_t(x)=\lim_{k\to\fz}f\ast P_t\ast \phi_{\vez_k}(x)
=f\ast P_t(x).$$
From this, we further deduce that, for all $x\in\rn$,
\begin{eqnarray*}
f_P^\ast(x)&&=\sup_{|y-x|<t}|f\ast P_t(y)|
=\sup_{|y-x|<t}\lim_{k\to\fz}|f_{\vez_k}\ast P_t(y)|\\
&&\le\lim_{k\to\fz}\sup_{|y-x|<t}|f_{\vez_k}\ast P_t(y)|
=\lim_{k\to\fz}(f_{\vez_k})_P^\ast(x).
\end{eqnarray*}
By this and the Fatou lemma (\cite[Theorem 2.61]{cfbook}), we conclude that
\begin{eqnarray*}
\|f_P^\ast\|_{\vlp}
&&\le \lf\|\lim_{k\to\fz}(f_{\vez_k}\ast P_t)^\ast\r\|_\vlp
\le \liminf_{k\to\fz}\|(f_{\vez_k})_P^\ast\|_\vlp\ls A_4.
\end{eqnarray*}
Therefore, $f\in\vhs$ and $\|f\|_{\vhs}\ls A_4$, which completes the
proof of Theorem \ref{t-rtc2}.
\end{proof}

\smallskip

\noindent{\bf Acknowledgements.} The authors would like to express their
deep thanks to Professor Yoshihiro Sawano for several useful conversations
on the subject of this article.

\bigskip

\noindent  Dachun Yang and Ciqiang Zhuo (Corresponding author)

\medskip

\noindent  School of Mathematical Sciences, Beijing Normal University,
Laboratory of Mathematics and Complex Systems, Ministry of
Education, Beijing 100875, People's Republic of China

\smallskip

\noindent {\it E-mails}: \texttt{dcyang@bnu.edu.cn} (D. Yang)

\hspace{0.98cm} \texttt{cqzhuo@mail.bnu.edu.cn} (C. Zhuo)

\vspace{0.5cm}

\noindent  Eiichi Nakai

\medskip

\noindent Department of Mathematics, Ibaraki University, Mito, Ibaraki 310-8512, Japan

\smallskip

\noindent {\it E-mail}: \texttt{enakai@mx.ibaraki.ac.jp} (E. Nakai)

\end{document}